\documentclass[11pt]{amsart}

\usepackage{amsmath,amsfonts,amssymb,amsthm}
\usepackage[all]{xy}
\usepackage{tikz}

\usepackage{bm}

\renewcommand{\bm}{}

\usetikzlibrary{arrows}
\usetikzlibrary{decorations.markings,decorations.pathmorphing,arrows}
\pgfarrowsdeclare{mytip}{mytip}{%
  \setlength{\arrowsize}{1pt}
  \addtolength{\arrowsize}{.5\pgflinewidth}
  \pgfarrowsrightextend{0}
  \pgfarrowsleftextend{-5\arrowsize}
}{%
  \setlength{\arrowsize}{1pt}
  \addtolength{\arrowsize}{.5\pgflinewidth}
  \pgfpathmoveto{\pgfpoint{-5\arrowsize}{1.5\arrowsize}}
  \pgfpathlineto{\pgfpointorigin}
  \pgfpathlineto{\pgfpoint{-5\arrowsize}{-1.5\arrowsize}}
  \pgfusepathqstroke
}
\tikzset{myarrow/.style={ decoration={bent,aspect=0.3, markings,mark=at
  position 0.5 with {\arrow[scale=1.2]{latex'}}}, postaction=decorate}}
\tikzset{myarrowshort/.style={ decoration={bent,aspect=0.3, markings,mark=at
  position 0.3 with {\arrow[scale=1.2]{latex'}}}, postaction=decorate}}
\tikzset{myarrowshorter/.style={ decoration={bent,aspect=0.3, markings,mark=at
  position 0.2 with {\arrow[scale=1.2]{latex'}}}, postaction=decorate}}
\tikzset{->-/.style={decoration={markings, mark=at position #1 with
  {\arrow{>}}},postaction={decorate}}}

\tikzset{my_dot/.style={fill, circle, inner sep=0pt,minimum size=1.5pt}}
\tikzset{my_node/.style={fill, circle, inner sep=0pt,minimum size=3pt}}

\theoremstyle{plain}
  \newtheorem{thm}{Theorem}[section]
  
  \newtheorem{lem}[thm]{Lemma}
  \newtheorem{cor}[thm]{Corollary}

\theoremstyle{definition}
  \newtheorem{defn}[thm]{Definition}
  
  \newtheorem{rem}[thm]{Remark}

\def\Z{\mathbb{Z}}

\DeclareMathOperator{\Div}{Div}
\DeclareMathOperator{\Jac}{\mathcal{S}}
\DeclareMathOperator{\Pic}{Pic}
\DeclareMathOperator{\T}{\mathcal{T}}
\DeclareMathOperator{\E}{\mathcal{E}}

\def\epsilon{\varepsilon}
\def\del{\partial}

\title{Sandpiles, spanning trees, and plane duality}
\author[M. Chan, D. Glass, M. Macauley, D. Perkinson, C. Werner, Q. Yang]{Melody Chan, Darren Glass, Matthew Macauley, \\
David Perkinson, Caryn Werner, Qiaoyu Yang}

\begin{document}

\begin{abstract} 
Let $G$ be a connected, loopless multigraph.  The sandpile group of $G$ is a
finite abelian group associated to $G$ whose order is equal to the number of
spanning trees in $G$.  Holroyd et al.~used a dynamical process on graphs called
rotor-routing to define a simply transitive action of the sandpile group of $G$
on its set of spanning trees.  Their definition depends on two pieces of
auxiliary data: a choice of a ribbon graph structure on $G$, and a choice of a
root vertex.  Chan, Church, and Grochow showed that if $G$ is a planar ribbon
graph, it has a canonical rotor-routing action associated to it, i.e., the
rotor-routing action is actually independent of the choice of root vertex.

It is well-known that the spanning trees of a planar graph $G$ are in canonical
bijection with those of its planar dual $G^*$, and furthermore that the sandpile
groups of $G$ and $G^*$ are isomorphic.  Thus, one can ask: are the two
rotor-routing actions, of the sandpile group of $G$ on its spanning trees, and
of the sandpile group of $G^*$ on its spanning trees, compatible under plane
duality? In this paper, we give an affirmative answer to this question, which
had been conjectured by Baker.
\end{abstract}

\maketitle
%%==========================================================

\section{Introduction}

Let $G$ be a connected multigraph with no loop edges.  The {\em sandpile group} of $G$ is a finite abelian group %associated to $G$
whose order is equal to the number of spanning trees in $G$; it is the group of {\em degree zero divisors} of $G$ modulo the equivalence relation generated by {\em lending moves}.  (We will recall all relevant definitions in Section \ref{S:prelim}.)

In \cite{rotor}, Holroyd, Levine, Meszar\'os, Peres, Propp, and Wilson use a dynamical process on graphs called {\em rotor-routing} to define a simply transitive action of the sandpile group of $G$  on its set of spanning trees.  Rotor-routing itself was introduced in \cite{PDDK} under the name ``Eulerian walkers'' and has been rediscovered several times in different fields: see \cite{rotor} for a concise history of the topic.

The definition of the rotor-routing action on $G$ given in \cite{rotor} involves two pieces of auxiliary data.  First, the action is defined with respect to a choice of a root vertex  $v\in V(G)$, or {\em basepoint}.  Second, it depends on a {\em ribbon graph} structure on $G$: a choice of a cyclic ordering of the set of edges incident to each vertex $v$.  Note that such a choice of cyclic orders defines an embedding of $G$ on some closed, oriented surface $S$, in which all cyclic orders correspond to a positive orientation, say, with respect to $S$.  We say that $G$ is a {\em planar} ribbon graph if $S$ is just a sphere, i.e., if the chosen ribbon structure equips~$G$ with an embedding into the plane.

A recent paper of Chan, Church, and Grochow~\cite{CCG} answers a question of J.~Ellenberg~\cite{MO} by proving that the rotor-routing action does not depend on the choice of basepoint if and only if $G$ is a planar ribbon graph.  This result is somewhat surprising, and as a nice consequence of it, we may henceforth refer to {\em the} rotor-routing action on a planar ribbon graph, without further reference to a choice of basepoint.

Any graph $G$ embedded in the plane has a planar dual graph $G^*$ whose spanning trees are in canonical bijection with those of $G$.  Moreover, the sandpile groups of $G$ and $G^*$ are, up to sign, canonically isomorphic \cite{bhn} (see also \cite{CR}). Thus, one would hope that the two rotor-routing actions, of the sandpile group of $G$ on the set $\T(G)$ of its spanning trees, and of the sandpile group of $G^*$ on its spanning trees, are compatible.

This was, in fact, exactly the conjecture suggested to us by M.~Baker.
In this paper, we provide a proof of Baker's conjecture on the compatibility of the rotor-routing action of the sandpile group with plane duality.  See Theorem~\ref{thm:main} for the precise statement, and see Figure~\ref{fig:intro} for an example illustrating the result.

We begin with preliminary definitions on the sandpile group and rotor-routing in Section \ref{S:prelim}.
The proof of our main result occupies Section \ref{S:main}.  The key idea of our proof is the {\em angle} betweeen two spanning trees $T$ and $T'$ of~$G$: see Definition~\ref{def:angle}.  The angle from $T$ to $T'$ remembers the element of the sandpile group that takes $T$ to $T'$ under rotor-routing.  On the other hand, we are able to show using a direct geometric argument that the angle is compatible with plane duality, so the main theorem follows.

We would also like to refer the reader to the recent preprint \cite{BW}, which arrives at another proof of Theorem \ref{thm:main} via a completely different route.
In that paper, Baker and Wang prove that the bijections obtained by Bernardi in  \cite[Theorem~45]{bernardi} give rise to another simply transitive action of the sandpile group on the spanning trees of a ribbon graph $G$ with a fixed root vertex.  They show that this action is compatible with plane duality and that it coincides with the rotor-routing action when $G$ is planar.  It would be interesting to study the relationship between these two approaches further.
\bigskip

\noindent {\bf Acknowledgments.} This work grew out of discussions at the workshop ``Generalizations of chip-firing and the critical group'' at the American Insitute of Mathematics (AIM) in Palo Alto, July 8--12, 2013.  The authors would like to thank the organizers of that conference (L.~Levine, J.~Martin, D.~Perkinson, and J.~Propp) as well as AIM and its staff, and M.~Baker for suggesting the conjecture that led to Theorem \ref{thm:main} and for comments on an earlier draft.  We thank Collin Perkinson for help with proofreading. MC was supported by NSF award number 1204278.  MM was supported by NSF grant DMS-1211691.

\begin{figure}
\begin{tikzpicture}[my_node/.style={fill, circle, inner sep=0pt,minimum size=2pt}, scale=0.78]
  \begin{scope}[shift={(0,0)},scale=1.0]
% graph G
\node[my_node,label=left:$x$] (x) at (-1,0) {};
\node[my_node,label=right:$y$] (y) at (1,0) {};
\node[my_node,label=below:$z$] (z) at (0,-1.4) {};
\node[my_node,label=above:$w$] (w) at (0,1.4) {};
\draw[ultra thick] (x)--(w);
\draw[ultra thick] (x)--(y);
\draw[ultra thick] (x)--(z);
\draw (w)--(y);
\draw (z)--(y);
\draw[font=\boldmath] (-1.5,1) node {$G$};
% dual graph G*
\node[my_node,label=above:$a$] (a) at (0,0.5) {};
\node[my_node,label=below:$b$] (b) at (0,-0.5) {};
\node[my_node,label=right:$c$] (c) at (2,0) {};
\draw[ultra thick,color=gray] (a) to [bend left=50] (c);
\draw[ultra thick,color=gray] (b) to [bend right=50] (c);
\draw[color=gray] (a)--(b);
\draw[color=gray] (a) to[out=165,in=-80] (-1.0,1.4) to[out=85,in=185] (0,2.3) to[out=-1,in=100] (c);
\draw[color=gray] (b) to[out=-165,in=80] (-1.0,-1.4) to[out=-85,in=-185] (0,-2.3) to[out=1,in=-100] (c);
\draw[font=\boldmath,color=gray] (2.0,1.8) node {$G^*$};
\end{scope}
% rotor-routing T
\begin{scope}[shift={(4,3)},scale=1.0]
\node[my_node] (x) at (-1,0) {};
\node[my_node] (y) at (1,0) {};
\node[my_node] (z) at (0,-1.4) {};
\node[draw,shape=circle,inner sep=2pt] (w) at (0,1.4) {};
\draw[->-=0.6,ultra thick] (w)--(x);
\draw[->-=0.6,ultra thick] (y)--(x);
\draw[->-=0.6,ultra thick] (z)--(x);
\draw (w)--(y);
\draw (z)--(y);
\draw (0,-2) node {$T$};
\end{scope}
\begin{scope}[shift={(7,3)},scale=1.0]
\node[my_node] (x) at (-1,0) {};
\node[draw,shape=circle,inner sep=2pt] (y) at (1,0) {};
\node[my_node] (z) at (0,-1.4) {};
\node[my_node] (w) at (0,1.4) {};
\draw[->-=0.6,ultra thick] (w)--(y);
\draw[->-=0.6,ultra thick] (y)--(x);
\draw[->-=0.6,ultra thick] (z)--(x);
\draw (w)--(x);
\draw (z)--(y);
\end{scope}
\begin{scope}[shift={(10,3)},scale=1.0]
\node[my_node] (x) at (-1,0) {};
\node[my_node] (y) at (1,0) {};
\node[my_node] (z) at (0,-1.4) {};
\node[draw,shape=circle,inner sep=2pt] (w) at (0,1.4) {};
\draw[->-=0.5,ultra thick] (w)--(y);
\draw[->-=0.5,ultra thick] (y)--(w);
\draw[->-=0.6,ultra thick] (z)--(x);
\draw (x)--(y);
\draw (w)--(x);
\draw (z)--(y);
\end{scope}
\begin{scope}[shift={(13,3)},scale=1.0]
\node[draw,shape=circle,inner sep=2pt] (x) at (-1,0) {};
\node[my_node] (y) at (1,0) {};
\node[my_node] (z) at (0,-1.4) {};
\node[my_node] (w) at (0,1.4) {};
\draw[->-=0.6,ultra thick] (w)--(x);
\draw[->-=0.6,ultra thick] (y)--(w);
\draw[->-=0.6,ultra thick] (z)--(x);
\draw (w)--(x);
\draw (z)--(y);
\draw (x)--(y);
\draw (-0.1,-2) node {$[w\!-\!x]\cdot T$};
\end{scope}
% rotor-route T^*
\begin{scope}[shift={(4,-2.5)},scale=0.8]
\node[my_node] (a) at (0,0.5) {};
\node[my_node] (b) at (0,-0.5) {};
\node[draw,shape=circle,inner sep=2pt] (c) at (2,0) {};
\draw[->-=0.6,ultra thick,color=gray] (c) to [bend right=50] (a);
\draw[->-=0.6,ultra thick,color=gray] (b) to [bend right=50] (c);
\draw[color=gray] (a)--(b);
\draw[color=gray] (a) to[out=165,in=-80] (-1.0,1.4) to[out=85,in=185] (0,2.3) to[out=-1,in=100] (c);
\draw[color=gray] (b) to[out=-165,in=80] (-1.0,-1.4) to[out=-85,in=-185] (0,-2.3) to[out=1,in=-100] (c);
\draw[color=black!80] (0.1,3) node {$T^*$};
\end{scope}
\begin{scope}[shift={(8,-2.5)},scale=0.8]
\node[my_node] (a) at (0,0.5) {};
\node[draw,shape=circle,inner sep=2pt] (b) at (0,-0.5) {};
\node[my_node] (c) at (2,0) {};
\draw[->-=0.5,ultra thick,color=gray] (b) to [bend right=50] (c);
\draw[->-=0.5,ultra thick,color=gray] (c) to [bend left=50] (b);
\draw[color=gray] (a) to [bend left=50] (c);
\draw[color=gray] (a)--(b);
\draw[color=gray] (a) to[out=165,in=-80] (-1.0,1.4) to[out=85,in=185] (0,2.3) to[out=-1,in=100] (c);
\draw[color=gray] (b) to[out=-165,in=80] (-1.0,-1.4) to[out=-85,in=-185] (0,-2.3) to[out=1,in=-100] (c);
\end{scope}
\begin{scope}[shift={(12,-2.5)},scale=0.8]
\node[draw,shape=circle,inner sep=2pt] (a) at (0,0.5) {};
\node[my_node] (b) at (0,-0.5) {};
\node[my_node] (c) at (2,0) {};
\draw[->-=0.6,ultra thick,color=gray] (c) to [bend left=50] (b);
\draw[color=gray] (a) to [bend left=50] (c);
\draw[->-=0.6,ultra thick,color=gray] (b)--(a);
\draw[color=gray] (a) to[out=165,in=-80] (-1.0,1.4) to[out=85,in=185] (0,2.3) to[out=-1,in=100] (c);
\draw[color=gray] (b) to[out=-165,in=80] (-1.0,-1.4) to[out=-85,in=-185] (0,-2.3) to[out=1,in=-100] (c);
\draw[color=black!80] (1.2,3) node {$[c\!-\!a]\cdot T^*$};
\end{scope}
\end{tikzpicture}
\caption{This figure shows the result of applying the element $[w\!-\!x]$ of $\Jac(G)$ to the spanning tree $T$ and, in the bottom row, the result of applying the element $[c\!-\!a]$ of $\Jac(G^*)$ to $T^*$.  The graph $G$ has all rotors oriented clockwise relative to the page, and its planar dual $G^*$ has all rotors oriented counterclockwise.  We chose $x$ and $a$ as our basepoints of $G$ and $G^*$ for the respective computations.
The isomorphism $\Jac(G)\cong \Jac(G^*)$ identifies $[w\!-\!x]$ and $[c\!-\!a]$, so the trees $[w\!-\!x]\cdot T$ and $[c\!-\!a]\cdot T^*$ must be dual trees, as shown on the right.  }
\label{fig:intro}
\end{figure}
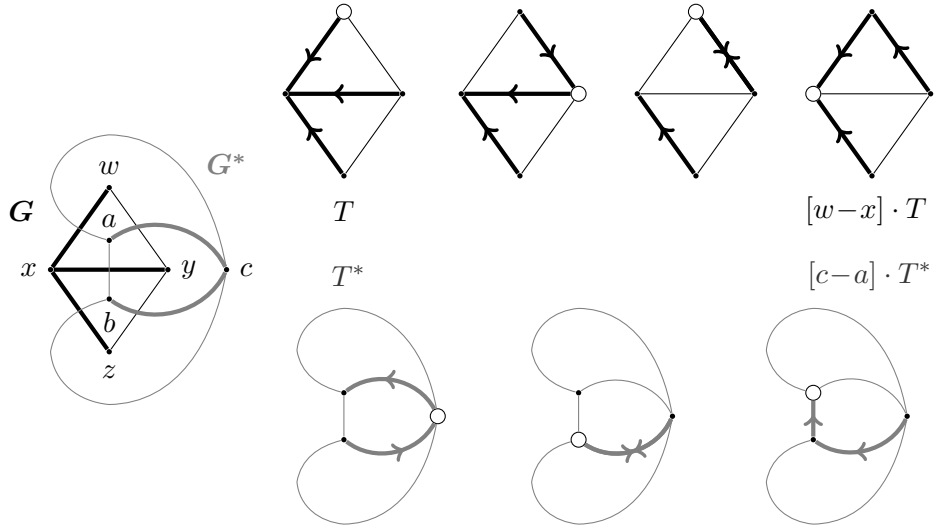

%%==========================================================
\section{Preliminaries}\label{S:prelim}
%%==========================================================
\subsection{The sandpile group}
Let $G=(V,E)$ be a finite connected loopless multigraph with vertex set $V$ and edge
multiset $E$.  The set of {\em divisors} on~$G$ is
the free abelian group on the vertices: $\Div(G) = \Z V$. We imagine a divisor
$D=\sum_{v\in V} a_v\,v$ to be an assignment of $D(v):=a_v$ chips to each
vertex $v$, keeping in mind that this number may be negative.  We write $\Div^0(G) $ for the subgroup
of divisors whose net number of chips $\sum D(v)$ is zero.

A {\em lending move} by a vertex $v$ consists of removing $\deg(v)$ chips from $v$
and distributing them along incident edges to the vertices neighboring $v$.
In other words, letting $n(v,w)$ denote the number of edges between $v$ and $w$, a lending move
by $v$ performed on a divisor $D$ produces a divisor $D'$ given by
\[
D'(w)=
\begin{cases}
  D(w) + n(v,w)&\text{if $w\neq v$}\\
  D(v)-\deg(v)&\text{if $w=v$}.
\end{cases}
\]
Notice that lending moves do not change the total number of chips in a divisor.  Divisors $D$ and $D'$ are {\em linearly equivalent}, denoted $D\sim D'$, if one can be obtained from the other by a sequence of lending moves at various vertices.
The {\em sandpile group} of $G$ is
\[
\Jac(G) = \Div^0(G)/\!\sim.
\]
The sandpile group of a graph is also variously known as the {\em Jacobian} of $G$, the {\em Picard group} $\Pic^0(G)$, or the {\em critical group} of $G$.

\subsection{Integral cuts and cycles} Fix an arbitrary orientation on the edges
$E$, and let $\Z E$ be the free abelian group on these oriented edges.  If
$e=\{u,v\}\in E$ is given the orientation $(u,v)$, we write
$e^+=\mathrm{head}(e)=v$ and $e^-=\mathrm{tail}(e)=u$.  We identify $-e$ with
the oppositely oriented edge $(v,u)$.  Each directed cycle on the underlying
undirected graph $G$ may be thought of as an element of $\Z E$, and the $\Z$-linear
span of these cycles in $\Z E$ is the {\em integral cycle space} for $G$, which we
denote by $\mathcal{C}$.

Next, for any subset $U\subset V$, the collection of all edges joining
a vertex of~$U$ to a vertex of $V\setminus U$ is called a {\em cut}.  By directing
all of these edges from vertices in $U$ to vertices in $V\setminus U$, we can identify this
cut with an element of~$\Z E$.
If $U$ consists of single vertex $v$, this cut is
called a {\em vertex cut} at $v$.
The integer span of all cuts is the {\em
integral cut space} for $G$ and is denoted by~$\mathcal{C}^*$.
Note that the vertex cuts generate the cut space.

Define \[\E(G) = \Z E/(\mathcal{C}+\mathcal{C}^*).\]  We now identify $\E(G)$ with the sandpile group $\Jac(G)$, as follows.  Define the {\em boundary map} $\Z E\to\Div^0(G)$ by sending each edge $e$ to $e^+-e^-$.
The boundary map is surjective since $G$ is connected, and its kernel is exactly the cycle space of $G$, so it identifies $\Div^0(G)$ with $\Z E/\mathcal{C}$.  Now given $D\in\Div^0(G)$, let $D_v$ be the
boundary of a vertex cut at the vertex $v$.  Then $D+D_v$ is the divisor obtained
from $D$ by performing a lending move at $v$.
Therefore the boundary map induces an isomorphism
\begin{align*}
  \partial_G\colon \E(G)&\stackrel{\cong}{\rightarrow}\Jac(G)\\
  e\ &\mapsto[e^+-e^-],
\end{align*}
as was proved in \cite[Proposition 8]{bhn}.
We will sometimes write $\del$ instead of~$\del_G$ for short.

\subsection{Rotor-routing action on spanning trees}\label{subsec:rotor}
Fix a {\em ribbon graph} structure on $G$, i.e., for each vertex $v$, fix a
cyclic ordering of the edges incident to~$v$.  Fix a vertex $q$. A {\em rotor
configuration with basepoint $q$} is the choice for each vertex $v\neq q$ of an
edge, $\rho(v)$, incident to~$v$.  We orient each edge $\rho(v)$ so that its
tail is $v$.

Let $D$ be a divisor on $G$, thought of as a chip configuration on
$G$, and let~$\rho$ be a rotor configuration with basepoint $q$.  We now recall
the {\em rotor-routing} process, by which a divisor $D$ transforms $\rho$ into a new
rotor configuration~$\rho'$.
{\em Firing} a vertex $v$ consists of updating~$\rho$ by replacing $\rho(v)$ with the
next edge in the cyclic ordering of edges at~$v$, then removing a chip from $v$
and placing it at the other end of the new edge $\rho(v)$.
Note that firing $v$ a total of
$\deg(v)$ times does not change the original rotor configuration, but transforms
$D$ by a lending move at $v$.
Now, every divisor $D$ on $G$ is linearly equivalent to a divisor $D'$ with $D'(v)\geq0$ for all $v\neq q$, see e.g.~\cite[Proposition 3.1]{bn}.  From that point, \cite{rotor} shows
that solely through vertex firings, all chips may be routed into $q$, and the rotor
configuration at the end of this process depends solely on the divisor class of
$D$.

Let $\T(G)$ denote the set of spanning trees of $G$.  Rooting
$T\in\T(G)$ at~$q$ uniquely determines a rotor configuration~$\rho_T$:  for
each vertex $v\neq q$, set~$\rho_T(v)$ to be the edge incident to $v$ on the
path in~$T$ from $v$ to $q$.  Given a divisor class $[D]\in\Jac(G)$, use the
rotor routing process to route all chips into~$q$ (at which point, all chips
will be gone since $\deg(D)=0$).  It is shown in \cite{rotor} that the resulting rotor
configuration is a spanning tree, directed into~$q$.  Call the underlying
undirected spanning tree $[D]\cdot T$.  Then according to \cite{rotor}, the resulting map
\begin{align*}
\mu_G\colon \Jac(G) \times \T(G) &\rightarrow \T(G)\\
([D],T)\quad&\mapsto [D]\cdot T
\end{align*}
is a simply transitive action of $\Jac(G)$ on $\T(G)$.

\subsection{Planar duality}  Now suppose that $G=(V,E)$ is a planar ribbon graph, and let
$G^*=(V^*,E^*)$ be its {\em planar dual graph}, whose vertices are the faces of $G$ and whose edges cross the edges of~$G$. We shall assume throughout that both $G$ and $G^*$ are loopless, i.e.,~$G$ has neither bridges nor loops.  We write $e^*$ for the edge of $G^*$ crossing the edge $e$ of $G$.  Each spanning tree of $G$ determines a spanning tree of $G^*$: namely, there is a natural bijection
\[\delta\colon \T(G) \stackrel{\cong}{\longrightarrow} \T(G^*)\]
sending $T$ to the tree $T^* = \{e^*\in E^*: e\in E\setminus T\}$.

Let us call
the orientation of the plane that agrees with the cyclic orderings of $G$ {\em clockwise}.
Then we fix once and for all the following {\em planar dual ribbon graph structure} on $G^*$: take the cyclic orderings of the edges at the vertices of $G^*$ to be {\em counter-clockwise}
with respect to the plane.

In order to define $\Z E$, we fixed an arbitrary orientation of the edges of~$G$.  To define $\Z E^*$, we will now choose a compatible orientation on the edges of~$G^*$.  For an oriented edge $e$ of $G$, let $e'$ (respectively $e''$) denote the edge at $v=e^-$ before (respectively after) $e$ in the cyclic order at $v$.  Now, call the face between $e'$ and $e$ at $v$ the face {\em before} $e$, and call the face between $e$ and $e''$ at $v$ the face {\em after} $e$.  Then we orient $e^*$ so that its head is the face of $G$ before $e$, and its tail is the face of $G$ after $e$.
For example, in Figure~\ref{fig:intro}, with the rotors of $G$ oriented clockwise relative to the page, suppose $e$ is the directed edge from $x$ to $y$.  Then $e^*$ is the directed edge in $G^*$ from $b$ to $a$.

Since directed cycles of $G$ are directed cuts of $G^*$ and vice versa, mapping each edge to its dual produces an isomorphism $\E(G)\cong\E(G^*)$, and hence we get an isomorphism $\phi$ of sandpile groups labeled as in the following commutative diagram:
 \begin{align*}
    \xymatrix{\ar[r]^{\cong} \E(G) \ar[d]_{\partial_G} &
      \E(G^*)\ar[d]^{\partial_{G^*}}\\
      \Jac(G)\ar[r]^{\phi}& \Jac(G^*).}
  \end{align*}

%%==========================================================
\section{Compatibility of rotor-routing with duality}\label{S:main}
%%==========================================================

%%----------------------------------------------------------
Let $G$ be any planar ribbon graph such that both $G$ and its dual $G^*$ are loopless.  In the previous section, we established an isomorphism $\phi\colon \Jac(G) \rightarrow \Jac(G^*)$ that depended on a single global choice of orientation of the $E^*$ derived from the orientation $E$.  With respect to this choice, we may now state the main theorem of the paper:

\begin{thm}\label{thm:main}
The diagram
\begin{align*}
\xymatrix{\Jac(G)\times \T(G) \ar[r]^-{\mu_G} \ar[d]_{\phi \times \delta} & \T(G) \ar[d]^{\delta} \\
\Jac(G^*) \times \T(G^*) \ar[r]^-{\mu_{G^*}} & \T(G^*)
}
\end{align*}
commutes.  In other words, the rotor-routing action is compatible with plane duality.
\end{thm}

In the rest of this section, we prove Theorem~\ref{thm:main}.
We begin with a topological definition of the angle between two spanning trees; this definition applies to all ribbon graphs, not just planar ones, and is the key idea in our proof of Theorem~\ref{thm:main}.

Suppose $G$ is any ribbon graph, and let  $e$ and $e'$ be directed edges emanating from a vertex $u$.  Suppose that in the cyclic order starting from $e=e_0$, the edges between $e$ and $e'$ are $e_0,e_1,\dots, e_k$ where $e_k=e'$, all directed outward from $u$. Define the {\em angle} between $e$ and $e'$ at $u$ by
\[
\angle^u(e,e')=\sum_{i=1}^k \del e_i\in \Jac(G).
\]
Recall that $\del$ denotes the boundary map sending a directed edge $e$ to the element $[e^+-e^-]\in\Jac(G)$.  Note that the sum includes $e'$ but not $e$.

\begin{figure}[h]
\begin{center}
\begin{tikzpicture}
    \draw[->,>=mytip] (0:0) -- (0:2) node[label={[label distance=1mm]0:{$e_0$}}]{};
  \draw[->,>=mytip] (0:0) -- (-20:2) node[label={[label distance=1mm]-20:{$e_1$}}]{};
  \draw[->,>=mytip] (0:0) -- (-40:2) node[label={[label distance=1mm]-40:{$e_2$}}]{};
  \draw[dotted,thick,-latex] (-56:1.8) arc (-56:-74:1.8);
  \draw[->,>=mytip] (0:0) -- (-90:2) node[label={[label distance=1mm]-90:{$e_k$}}]{};
\end{tikzpicture}
\end{center}
\caption{$\angle^u(e_0,e_k)=\del e_1+\dots+\del e_k$.}
\end{figure}

\begin{lem}\label{lem:dual angle}
Suppose $G$ is a planar ribbon graph, and let $e_0,\ldots,e_k$ be consecutive outgoing edges from some vertex $u$ in the cyclic order at $u$.  For $i=0, \ldots, k$, let $r_i$ be the face of $G$, equivalently the vertex of $G^*$, lying to the right of $e_i$ (with respect to the cyclic order at $u$).  Then
\[\phi(\angle^u(e_0,e_k)) = [r_0 - r_k] \in \Jac(G^*).\]
\end{lem}

\begin{proof}
By definition, $\phi(\del e_i) = [r_{i-1}-r_i]$.  By linearity, it follows that $\phi(\angle^u(e_0,e_k))$ is the telescoping sum $[(r_0-r_1) + (r_1 -r_2)+\cdots + (r_{k-1}-r_k)]$, proving the claim.
\end{proof}

\begin{defn}\label{def:angle} Let $G$ be an arbitrary ribbon graph, and let $T$ and $T'$ be two spanning trees of $G$.  Let $v \in V$ be any vertex.  As in \S \ref{subsec:rotor}, let $\rho_T$ and~$\rho_{T'}$ be the rotor configurations based at $v$ arising from orienting $T$ and $T'$ towards $v$.

The {\em angle} between~$T$ and $T'$ based at $v$, denoted $\angle_v(T,T')$, is the sum of the angles between their edges at each non-root vertex.  That is,

\[ \angle_v(T,T') := \sum_{u\in V\setminus \{v\}} \angle^u (\rho_T(u), \rho_{T'}(u) ) \in \Jac(G).\]
\end{defn}

\begin{lem}\label{lem:activated edges}
Let $G$ be any ribbon graph, and let $T$ be a spanning tree of $G$.
For any vertex $v$ and any $[D]\in \Jac(G)$, we have \[\angle_v(T,[D]\cdot T)) = [-D].\]
Here, the rotor-routing action of $[D]$ on $T$ is computed with respect to the basepoint~$v$.
\end{lem}
\begin{proof} Without loss of generality, we may choose $D$ to be a chip configuration that is nonnegative at vertices other than $v$.  Consider the rotor-routing process that calculates $[D]\cdot T$.  We will say that the directed edge $(x,y)$ is {\em activated} if a chip is sent from vertex~$x$ to vertex $y$ during this  process.  Note that when the chip is fired, the chip configuration on the graph changes by $\del(x,y) = y-x$.  Since at the end of the rotor-routing process there are no chips left on the graph, it follows that
\[
[D] + \sum_e\del e = 0,
\]
where the sum is over the multiset of edges that have been activated during the process.

Next, we claim that the angle between $T$ and $[D]\cdot T$ is in fact equal to $\sum_e \del e$, where the sum is again over the multiset of activated edges.  This is because at each vertex $u \ne v$, the sum of the boundaries of all outgoing edges $e$ at $u$ is $0\in\Jac(G)$; after all, this sum corresponds to a lending move at $u$.  So the sum over all activated edges leaving $u$ is exactly the angle at~$u$ between the edge of $T$ leaving $u$ and that of $T'$, and the claim follows.  Summarizing, we have
\[
\angle_v(T,[D]\cdot T)) = \sum_e \del e = [-D].
\]
\end{proof}

\begin{cor}\label{cor:zero angle}
Let $G$ be any planar ribbon graph, and let $T$ and $T'$ be spanning trees of $G$ rooted at the same vertex $v$.  Then $\angle_v(T,T')=0$ if
  and only if $T=T'$.
\end{cor}
\begin{proof} Assume that $\angle_v(T,T')=0$, and let $[D]\in \Jac(G)$ take $T$ to $T'$ under the rotor-routing action with basepoint $v$.  It follows from \cite[Lemma 3.17]{rotor} that the element $[D]$ exists and is unique.  Then by Lemma~\ref{lem:activated edges}, $[D]=0$, so $T = T'$.  The converse is clear.
\end{proof}

\begin{rem}\label{rem: root dependence}
It follows from Lemma~\ref{lem:activated edges} and from \cite[Theorem 2]{CCG} that the notion of angle between trees for $G$ is independent of the choice of root vertex for the trees if and only if $G$ is a planar ribbon graph. Indeed, Lemma~\ref{lem:activated edges} shows that $\angle_v(T,T')$ is exactly the element of $\Jac(G)$ sending $T'$ to $T$ in the rotor-routing action based at $v$, and the rotor-routing action is basepoint-independent if and only if $G$ is a planar ribbon graph by \cite{CCG}.
Thus, if $G$ is planar, we will henceforth write $\angle(T,T')$ for the angle between $T$ and $T'$, computed with respect to any vertex.
\end{rem}

We can now prove our main lemma.

\begin{lem}\label{lem:duality}
Let $G$ be a planar ribbon graph, and
let $T$ and $T'$ be spanning trees of $G$. Then
\[
\phi(\angle(T,T'))=\angle(T^*,T'^*)\,.
\]
\end{lem}
\begin{proof}
Given a spanning tree $T$ and an edge $e$ not in $T$, we call the unique cycle $C(e)$ in $T\cup \{e\}$ the {\em fundamental cycle} of $e$ with respect to $T$.  We first note that there is a sequence of trees $T=T_0,T_1,\ldots,T_r=T'$ such that for each $j$, the trees $T_{j+1}$ and $T_j$ have exactly $n-1$ edges in common.  If $T=T'$ this statement is vacuously true. Otherwise, pick $e'\in T'\setminus T$; then the fundamental cycle of $e'$ with respect to $T$ must contain some edge $e \in T\setminus T'$.  Set $T_1 = T\cup\{e'\} \setminus \{e\}.$  Then $T_1$ and $T'$ have smaller symmetric difference, so repeating, we produce a sequence of spanning trees as desired.  It follows by induction that we may assume $T' = T\cup\{e'\}\setminus\{e\}$.

In fact, we may further assume, again by induction, that $e$ and $e'$ are edges incident to a common
  face of $G$.  Indeed, since $T^*\cup \{e^*\}\setminus \{e'^*\} = T'^*$ is acyclic, the fundamental cycle $C(e^*)$ of $e^*$ with respect to $T^*$ contains $e'^*$.  Now starting at $e^*$ and proceeding along the cycle $C(e^*)$ in either direction, let $e^*=e_0^*, e_1^*, \ldots, e_s^* = e'^*$ be the sequence of edges traversed.
Then \[T^*,~(T^*\cup \{e^*\})\!\setminus\!\{e_1^*\},~(T^*\cup
\{e^*\})\!\setminus\!\{e_2^*\},\ldots,~(T^*\cup \{e^*\})\!\setminus\!\{e_s^*\}\]
is a sequence of trees in $G^*$ such that the symmetric difference of any consecutive pair of trees consists of two edges of $G^*$ adjacent to the same vertex.  Now passing to $G$, we conclude that
\[T,~(T\cup \{e_1\})\!\setminus\!\{e\},~(T\cup \{e_2\})\!\setminus\!\{e\},\ldots,~(T\cup \{e'\})\!\setminus\!\{e\}\]
is a sequence of trees in $G$ such that the symmetric difference of any consecutive pair of trees consists of two edges of $G$ incident to the same face.

Thus from here on, we assume that $T' = (T\cup\{e'\})\setminus\{e\}$, where $e, e'\in E(G)$ are incident to a common face, which we call $f$. Write $e = xy$ and $e' = x'y'$ for vertices $x,y,x',y'$ of $V(G)$, such that $f$ is to the left of the edge~$e$ when it is traversed in the direction $x\rightarrow y$, and $f$ is to the right of the edge~$e'$ when it is traversed in the direction $x'\rightarrow y'$.
Write $\mathcal{C}$ for the fundamental cycle in
  $T\cup\{e'\}$; it is illustrated in Figure~\ref{fig:notation}.
(Here and throughout the rest of the proof, we assume a clockwise orientation on the rotors of $G$ simply in order to talk about the left and right sides of an edge freely.  For example, the face to the right of an oriented edge $e = (x,y)$ should be interpreted as the face coming in between $e$ and the edge after $e$ in the cyclic order at $x$.)

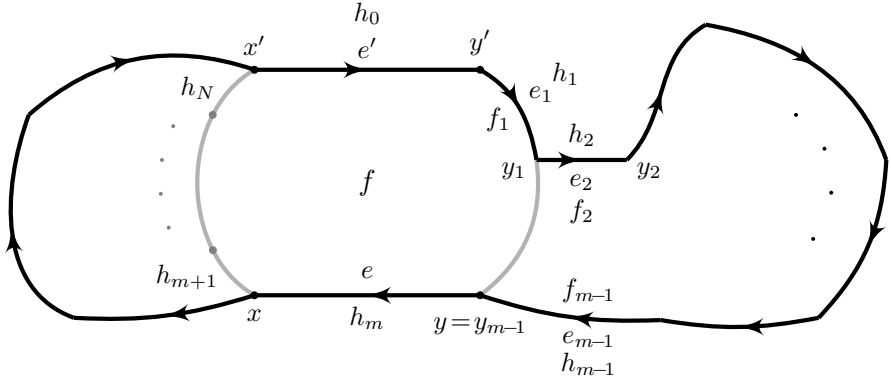
\begin{figure}
\begin{tikzpicture}[scale=3]
  \node at (1.08,0.78) {\small $f_1$};
  \node at (1.15,0.55) {\small $y_1$};
  \node at (1.27,0.89) {\small $e_1$};
  \node at (1.38,0.98) {\small $h_1$};
  \node at (1.45,0.71) {\small $h_2$};
  \node at (1.45,0.50) {\small $e_2$};
  \node at (1.45,0.37) {\small $f_2$};
  \node at (1.75,0.55) {\small $y_2$};
  \node at (1.48,0.02) {\small $f_{m\!-\!1}$};
  \node at (1.48,-0.2) {\small $e_{m\!-\!1}$};
  \node at (1.48,-0.31) {\small $h_{m\!-\!1}$};
%  \node at (1.74,1.1) {\small $h_3$};
%  \node at (1.96,1.0) {\small $e_3$};
%  \node at (2.1,0.95) {\small $f_3$};
%  \node at (1.97,1.295) {\small $y_3$};
  \node at (0.5,0.1) {\small $e$};
  \node at (0.5,-0.1) {\small $h_m$};
  \node at (-0.25,0.92) {\small $h_N$};
  \node at (-0.3,0.08) {\small $h_{m+1}$};
  \draw[ultra thick,myarrow] (1,0) to node[midway] {} (0,0);
  \draw[ultra thick,myarrow] (0,1) to node[black,midway,above] {\small $e'$} (1,1);
  \draw[black] (0,1)--(1,1);
  \node at (0.5,1.25) {\small $h_0$};
  % the face F
  \draw[ultra thick,gray!60] (0,1) to[bend right=60] (0,0);
  \draw[ultra thick,myarrow] (1,1) to[out=-30,in=100] (1.25,0.6);
  \draw[ultra thick,gray!60] (1.25,0.6) to[bend left=30] (1,0);
  \node at (0.5,0.5) {$f$};
  % right loop
  \draw[ultra thick,myarrow] (1.8,-0.11) to[bend left=10] (1,0);
  \draw[ultra thick,myarrow] (2.5,-0.1) to[bend left=10] (1.8,-0.11) ;
  \draw[ultra thick,myarrow] (2.8,0.6) to[bend left=20] (2.5,-0.1);
  \draw[ultra thick,myarrow] (2.0,1.2) to[out=-10,in=110] (2.8,0.6);
  \draw[ultra thick,myarrow] (1.65,0.6) to[out=45,in=210] (2.0,1.2);
  \draw[ultra thick,myarrow] (1.25,0.6) to (1.65,0.6);
  % left loop
  \draw[ultra thick,myarrow] (0,0) to[bend left=10] (-0.8,-0.1);
  \draw[ultra thick,myarrow] (-0.8,-0.1) to[out=160,in=250] (-1.0,0.8);
  \draw[ultra thick,myarrow] (-1.0,0.8) to[bend left=30] (0,1);

  \node[my_node,label=above:{\small $x'$}] at (0,1) {};
  \node[my_node,color=gray] at (-0.185,0.8) {};
  \node[my_node,color=gray] at (-0.184,0.2) {};
  \node[my_node,label=above:{\small $y'$}] at (1,1) {};
  \node[my_node,label=below:{\small $x$}] at (0,0) {};
  \node[my_node,label={[label distance=3pt]270:{\small $y\!=\!y_{m\!-\!1}$}}] at (1,0) {};

  \node[my_dot,color=gray] at (-0.38,0.3) {};
  \node[my_dot,color=gray] at (-0.415,0.45) {};
  \node[my_dot,color=gray] at (-0.41,0.6) {};
  \node[my_dot,color=gray] at (-0.36,0.75) {};

  \node[my_dot] at (2.4,0.8) {};
  \node[my_dot] at (2.526,0.65) {};
  \node[my_dot] at (2.556,0.455) {};
  \node[my_dot] at (2.476,0.25) {};

  %\node at (-1,1.2) {$T$};
\end{tikzpicture}

\caption{The fundamental cycle $\mathcal{C}$ of $T\cup\{e'\}$, shaded in black.}
\label{fig:notation}
\end{figure}

By Remark~\ref{rem: root dependence}, the calculation of the angle
$\angle(T,T')\in\Jac(G)$ is independent of the choice of root vertex.  Choose $x'$ as the root and orient $T$ and $T'$ towards $x'$. We wish to study the sum of the angles at each vertex $v\ne x'$ of~$G$ between the edges of $T$ and $T'$ that are outgoing from $v$.

Having rooted the trees at $x'$, we start by observing that the path between~$y$ and $y'$ in $T$ is directed from $y'$ to $y$, whereas in $T'$ it has the opposite orientation.  This is illustrated in Figure~\ref{fig:rooted trees}.  Furthermore, all other edges shared by $T$ and $T'$ have the same orientation.  Indeed, consider a vertex $v$ not on $\mathcal{C}$ and say its unique path in $T$ to $x'$ first meets $\mathcal{C}$ at $v'$; then the same path $v$--$v'$ in $T'$ must be an initial subpath of the unique path in~$T'$ from $v$ to $x'$, so in particular the edge leaving $v$ is unchanged.

\vspace{.5cm}

\begin{figure}[!htbp]
  \begin{center}
% the tree T rooted at x'
\begin{tikzpicture}[scale=2]
  \node at (-1.5,0.5) {$T$};
  % edges e and e'
  \node[above] at (0,1) {$x'$};
  \node[above] at (1,1) {$y'$};
  \node[below] at (0,0) {$x$};
  \node[below] at (1,0) {$y$};
  \draw[ultra thick,myarrow] (1,0) to node[midway,below] {$e$} (0,0);
  \draw[ultra thick,gray!60] (0,1) to node[black,midway,above] {$e'$} (1,1);
  % the face F
  \draw[ultra thick,gray!60] (0,1) to[bend right=60] (0,0);
  \draw[ultra thick,myarrow] (1,1) to[out=-30,in=100] (1.25,0.6);
  \draw[ultra thick,gray!60] (1.25,0.6) to[bend left=30] (1,0);
  \node at (0.5,0.5) {$f$};
  % right loop
  \draw[ultra thick,myarrow] (1.8,-0.11) to[bend left=10] (1,0);
  \draw[ultra thick,myarrow] (2.5,-0.1) to[bend left=10] (1.8,-0.11) ;
  \draw[ultra thick,myarrow] (2.8,0.6) to[bend left=20] (2.5,-0.1);
  \draw[ultra thick,myarrow] (2.0,1.2) to[out=-10,in=110] (2.8,0.6);
  \draw[ultra thick,myarrow] (1.65,0.6) to[out=45,in=210] (2.0,1.2);
  \draw[ultra thick,myarrow] (1.25,0.6) to (1.65,0.6);
  % left loop
  \draw[ultra thick,myarrow] (0,0) to[bend left=10] (-0.8,-0.1);
  \draw[ultra thick,myarrow] (-0.8,-0.1) to[out=160,in=250] (-1.0,0.8);
  \draw[ultra thick,myarrow] (-1.0,0.8) to[bend left=30] (0,1);

% the tree T' rooted at x'
  \begin{scope}[shift={(0,-2.0)}]
  \node at (-1.5,0.5) {$T'$};
  % edges e and e'
  \node[above] at (0,1) {$x'$};
  \node[above] at (1,1) {$y'$};
  \node[below] at (0,0) {$x$};
  \node[below] at (1,0) {$y$};
  \draw[ultra thick,gray!60] (1,0) to node[black,midway,below] {$e$} (0,0);
  \draw[ultra thick,myarrow] (1,1) to node[midway,above] {$e'$} (0,1);
  % the face F
  \draw[ultra thick,gray!60] (0,1) to[bend right=60] (0,0);
  \draw[ultra thick,myarrow] (1.25,0.6) to[out=100,in=-30] (1,1);
  \draw[ultra thick,gray!60] (1.25,0.6) to[bend left=30] (1,0);
  \node at (0.5,0.5) {$f$};
  % right loop
  \draw[ultra thick,myarrow] (1,0) to[bend right=10] (1.8,-0.11);
  \draw[ultra thick,myarrow] (1.8,-0.11) to[bend right=10] (2.5,-0.1);
  \draw[ultra thick,myarrow] (2.5,-0.1)to[bend right=20] (2.8,0.6);
  \draw[ultra thick,myarrow] (2.8,0.6) to[out=110,in=-10] (2.0,1.2);
  \draw[ultra thick,myarrow] (2.0,1.2) to[out=210,in=45] (1.65,0.6);
  \draw[ultra thick,myarrow] (1.65,0.6) to (1.25,0.6);
  % left loop
  \draw[ultra thick,myarrow] (0,0) to[bend left=10] (-0.8,-0.1);
  \draw[ultra thick,myarrow] (-0.8,-0.1) to[out=160,in=250] (-1.0,0.8);
  \draw[ultra thick,myarrow] (-1.0,0.8) to[bend left=30] (0,1);
\end{scope}
\end{tikzpicture}
  \end{center}
\caption{Parts of the trees $T$ and $T'$, rooted at the vertex $x'$.}\label{fig:rooted trees}
\end{figure}
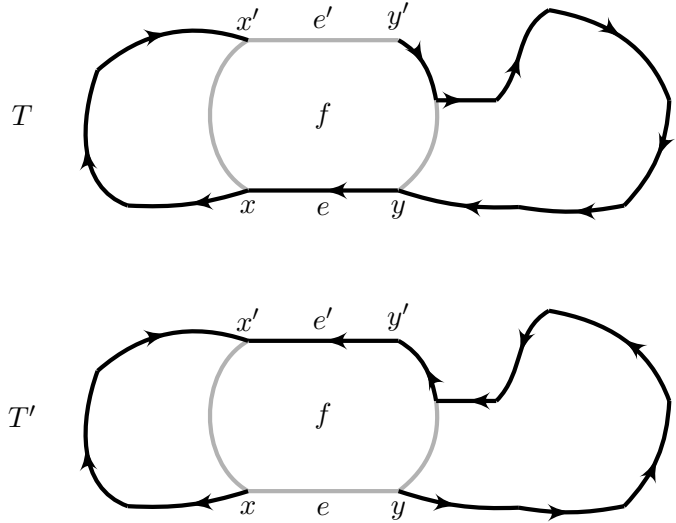

\vspace{.5cm}

Let us fix some notation before going further.  Write \[y' = y_0, e_1, y_1, e_2,\ldots, y_{m-1} = y\] for the sequence of vertices and directed edges in the $y'$--$y$ path in $T$.  For each directed edge $e_i$, we write $f_i$ (respectively $h_i$) for the face of $G$ to the right (respectively left) of $e_i$.

For convenience, we extend the notation above as follows.  We denote by~$h_0$ the face of $G$ to the left of $e'$ when oriented from $x'$ to $y'$, and we denote by $h_m$ the face of $G$ to the left of $e$ when oriented from $y$ to $x$.  Next, consider the path from $x$ to $x'$ that bounds $f$ and such that $f$ lies on its right. Call the faces on the {\em left} side of this $x$--$x'$ path $h_{m+1},\ldots,h_N$. See Figure \ref{fig:notation}.

Letting $e_0=e'$ and $e_m=e$, the angle between $T$ and $T'$ then is given by
\[\angle(T,T') = \sum_{i=0}^{m-1} \angle^{y_i} (e_{i+1}, e_i) \in \Jac(G),\]
where in each expression in the sum, we regard each edge as being oriented away from $y_i$ in turn.  Then by Lemma~\ref{lem:dual angle}, we have
\[\phi(\angle(T,T')) = (f_1 - h_0) + (f_2 - h_1) + \cdots+(f_{m-1}-h_{m-2}) +(f-h_{m-1}) \in\Jac(G^*).\]
The angle between $T$ and $T'$ is shown in~Figure~\ref{fig:angle}.  The signs indicate
$\phi(\angle(T,T'))\in\Jac(G^*)$.

\vspace{.5cm}

\begin{figure}[!h]
  \begin{center}
% angle between T and T' as element of E(G) and Jac(G^*)
\begin{tikzpicture}[scale=2.5]
  % edges e and e'
  \node[above] at (0,1) {$x'$};
  \node[above] at (1,1) {$y'$};
  \node[below] at (0,0) {$x$};
  \node[below] at (1,0) {$y$};
  \draw[ultra thick] (0,0) to node[midway,below] {$e$} (1,0);
  \draw[ultra thick] (0,1) to node[midway,above] {$e'$} (1,1);
  %\node[above] at (0.8,1) {$\bm{-}$};
  % the face F
  \draw[ultra thick,gray!60] (0,1) to[bend right=60] (0,0);
  \draw[ultra thick] (1,1) to[out=-30,in=100] (1.25,0.6);
  \draw[ultra thick,gray!60] (1.25,0.6) to[bend left=30] (1,0);
  \node at (0.5,0.5) {$f$};
  % right loop
  \draw[ultra thick] (1,0) to[bend right=10] (1.8,-0.11);
  \draw[ultra thick] (1.8,-0.11) to[bend right=10] (2.5,-0.1);
  \draw[ultra thick] (2.5,-0.1) to[bend right=20] (2.8,0.6);
  \draw[ultra thick] (2.8,0.6) to[out=110,in=-10] (2.0,1.2);
  \draw[ultra thick] (2.0,1.2) to[out=210,in=45] (1.65,0.6);
  \draw[ultra thick] (1.65,0.6) to (1.25,0.6);
  % left loop
  \draw[ultra thick] (0,0) to[bend left=10] (-0.8,-0.1);
  \draw[ultra thick] (-0.8,-0.1) to[out=160,in=250] (-1.0,0.8);
  \draw[ultra thick] (-1.0,0.8) to[bend left=30] (0,1);
  % angles
    % first
  \draw[ultra thick,-latex'] (1,0) to[bend right=8] (1.163,0.17);
%  \draw[ultra thick,-latex'] (1,0) to (1.3,0.15);
  \draw[ultra thick,-latex'] (1,0) to (1.37,0.07);
  \draw[ultra thick,myarrowshort] (1,0) to[bend right=10] (1.8,-0.11);
    % second
  \draw[ultra thick,-latex'] (1.8,-0.11) to (1.7,0.07);
  \draw[ultra thick,-latex'] (1.8,-0.11) to (1.9,0.07);
  \draw[ultra thick,myarrowshort] (1.8,-0.11) to[bend right=10] (2.5,-0.1);
    % third
  \draw[ultra thick,-latex'] (2.5,-0.1) to (2.35,0.08);
  \draw[ultra thick,myarrowshort] (2.5,-0.1)to[bend right=20] (2.8,0.6);
    % fourth
  \draw[ultra thick,-latex'] (2.8,0.6) to (2.6,0.5);
%  \draw[ultra thick,-latex'] (2.8,0.6) to (2.6,0.55);
  \draw[ultra thick,-latex'] (2.8,0.6) to (2.6,0.66);
  \draw[ultra thick,myarrowshorter] (2.8,0.6) to[out=110,in=-10] (2.0,1.2);
    % fifth
  \draw[ultra thick,-latex'] (2.0,1.2) to (2.05,1.0);
  \draw[ultra thick,myarrowshort] (2.0,1.2) to[out=210,in=45] (1.65,0.6);
    % sixth
  \draw[ultra thick,-latex'] (1.65,0.6) to (1.9,0.55);
  \draw[ultra thick,-latex'] (1.65,0.6) to (1.66,0.4);
  \draw[ultra thick,myarrow] (1.65,0.6) to (1.25,0.6);
    % seventh
  \draw[ultra thick,-latex'] (1.25,0.6) to (1.40,0.4);
  \draw[ultra thick,-latex'] (1.25,0.6) to[bend left=6] (1.249,0.4);
  \draw[ultra thick,myarrowshort] (1.25,0.6) to[out=100,in=-30] (1,1);
  \draw[ultra thick,myarrowshort] (1,1) to (0,1);
  % pluses and minuses
    % first
  \node at (1.00,0.14) {\tiny $\bm{+}$};
  \node at (1.20,-0.13) {\tiny $\bm{-}$};
    % second
  \node at (1.62,0.01) {\tiny $\bm{+}$};
  \node at (1.9,-0.21) {\tiny $\bm{-}$};
    % third
  \node at (2.27,0.025) {\tiny $\bm{+}$};
  \node at (2.69,-0.04) {\tiny $\bm{-}$};
    % fourth
  \node at (2.70,0.43) {\tiny $\bm{+}$};
  \node at (2.85,0.72) {\tiny $\bm{-}$};
    % fifth
  \node at (2.14,1.07) {\tiny $\bm{+}$};
  \node at (1.86,1.21) {\tiny $\bm{-}$};
    % sixth
  \node at (1.85,0.66) {\tiny $\bm{+}$};
  \node at (1.55,0.68) {\tiny $\bm{-}$};
    % seventh
  \node at (1.46,0.48) {\tiny $\bm{+}$};
  \node at (1.32,0.75) {\tiny $\bm{-}$};
    % eighth
  \node[above] at (0.8,1) {\tiny $\bm{-}$};
  \node[below] at (0.8,1) {\tiny $\bm{+}$};
\end{tikzpicture}
  \end{center}
  \caption{$\angle(T,T')\in\Jac(G)$ and $\phi(\angle(T,T'))\in\Jac(G^*)$, the former drawn with arrows and the latter drawn with plus and minus signs.}\label{fig:angle}
\end{figure}
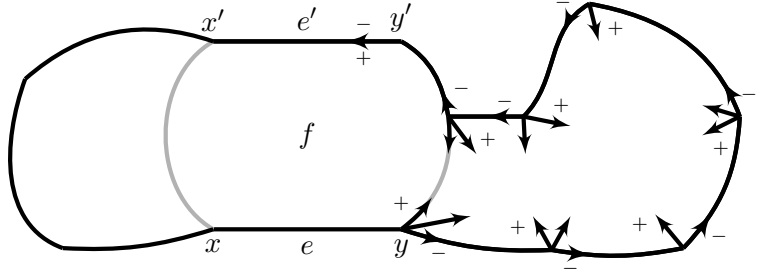

\vspace{.5cm}

Next,
consider the oriented cycle $C$ running from $x'$ to
$y'$, then along edges of $T$ from $y'$ to $x$, then along edges of $f$ back to
$x'$, as shown in~Figure~\ref{fig:cut of G*}.  The dual $C^*$ of $C$ is a cut of
$G^*$, so $\del_{G^*}(C^*) = 0 \in \Jac(G^*)$.  On the other hand,
\[\del_{G^*}(C^*) = (h_0-f) + (h_1-f_1) + \cdots + (h_{m-1}-f_{m-1}) + \sum_{i=m}^N (h_i-f).\]
The signs in Figure~\ref{fig:cut of G*} indicate $\del_{G^*}(C^*)\in\Jac(G^*)$.
\begin{figure}[!h]
  \begin{center}
% cut of G* dual to x',y',y,x cycle
\begin{tikzpicture}[scale=2.5]
  % edges e and e'
  \node[above] at (0,1) {$x'$};
  \node[above] at (1,1) {$y'$};
  \node[below] at (0,0) {$x$};
  \node[below] at (1,0) {$y$};
  \draw[ultra thick] (0,0) to node[midway,below] {$e$} (1,0);
  \draw[ultra thick] (0,1) to node[midway,above] {$e'$} (1,1);
  % the face F
  \draw[ultra thick] (0,1) to[bend right=60] (0,0);
  \draw[ultra thick] (1,1) to[out=-30,in=100] (1.25,0.6);
  \draw[ultra thick,gray!60] (1.25,0.6) to[bend left=30] (1,0);
  \filldraw (-0.232,0.67) circle (0.2mm);
  \filldraw (-0.232,0.33) circle (0.2mm);
  \node at (0.5,0.5) {$f$};
  % right loop
  \draw[ultra thick] (1,0) to[bend right=10] (1.8,-0.11);
  \draw[ultra thick] (1.8,-0.11) to[bend right=10] (2.5,-0.1);
  \draw[ultra thick] (2.5,-0.1) to[bend right=20] (2.8,0.6);
  \draw[ultra thick] (2.8,0.6) to[out=110,in=-10] (2.0,1.2);
  \draw[ultra thick] (2.0,1.2) to[out=210,in=45] (1.65,0.6);
  \draw[ultra thick] (1.65,0.6) to (1.25,0.6);
  % left loop
  \draw[ultra thick,gray!60] (0,0) to[bend left=10] (-0.8,-0.1);
  \draw[ultra thick,gray!60] (-0.8,-0.1) to[out=160,in=250] (-1.0,0.8);
  \draw[ultra thick,gray!60] (-1.0,0.8) to[bend left=30] (0,1);
  % pluses and minuses
  \node at (0.5,-0.23) {\tiny $\bm{+}$};
  \node at (0.5,0.08) {\tiny $\bm{-}$};
  \node at (1.4,-0.2) {\tiny $\bm{+}$};
  \node at (1.4,-0.01) {\tiny $\bm{-}$};
  \node at (2.13,-0.23) {\tiny $\bm{+}$};
  \node at (2.13,-0.06) {\tiny $\bm{-}$};
  \node at (2.8,0.2) {\tiny $\bm{+}$};
  \node at (2.6,0.2) {\tiny $\bm{-}$};
  \node at (2.6,1.05) {\tiny $\bm{+}$};
  \node at (2.45,0.92) {\tiny $\bm{-}$};
  \node at (1.72,0.95) {\tiny $\bm{+}$};
  \node at (1.9,0.9) {\tiny $\bm{-}$};
  \node at (1.45,0.7) {\tiny $\bm{+}$};
  \node at (1.45,0.51) {\tiny $\bm{-}$};
  \node at (1.24,0.9) {\tiny $\bm{+}$};
  \node at (1.05,0.8) {\tiny $\bm{-}$};
  \node at (0.5,1.27) {\tiny $\bm{+}$};
  \node at (0.5,0.92) {\tiny $\bm{-}$};
  \node at (-0.26,0.87) {\tiny $\bm{+}$};
  \node at (-0.04,0.85) {\tiny $\bm{-}$};
  \node at (-0.35,0.5) {\tiny $\bm{+}$};
  \node at (-0.15,0.5) {\tiny $\bm{-}$};
  \node at (-0.26,0.15) {\tiny $\bm{+}$};
  \node at (-0.08,0.18) {\tiny $\bm{-}$};
\end{tikzpicture}
  \end{center}
  \caption{The cycle $C$ in black and $\del_{G^*}(C^*)$.}\label{fig:cut of G*}
\end{figure}
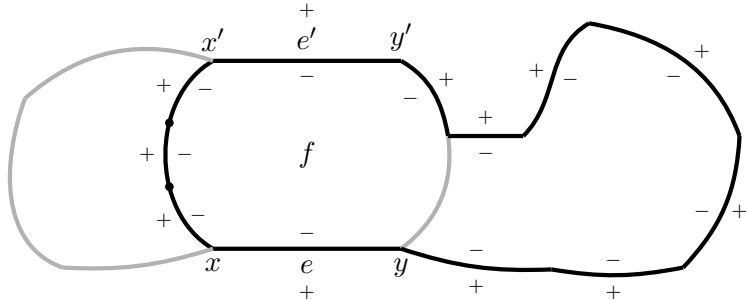

\vspace{.5cm}

Summing, we have
\[\phi(\angle(T,T')))+\del_{G^*}(C^*) =\sum_{i=m}^N (h_i-f). \]
This sum is shown in Figure~\ref{fig:T+cut}.

\vspace{.5cm}

\begin{figure}[!h]
  \begin{center}
% sum of angle(T,T*) and the cut of G* dual to x',y',y,x cycle
\begin{tikzpicture}[scale=2.5]
  % edges e and e'
  \node[above] at (0,1) {$x'$};
  \node[above] at (1,1) {$y'$};
  \node[below] at (0,0) {$x$};
  \node[below] at (1,0) {$y$};
  \draw[ultra thick] (0,0) to node[midway,below] {$e$} (1,0);
  \draw[ultra thick] (0,1) to node[midway,above] {$e'$} (1,1);
  % the face F
  \draw[ultra thick] (0,1) to[bend right=60] (0,0);
  \draw[ultra thick] (1,1) to[out=-30,in=100] (1.25,0.6);
  \draw[ultra thick,gray!60] (1.25,0.6) to[bend left=30] (1,0);
  \filldraw (-0.232,0.67) circle (0.2mm);
  \filldraw (-0.232,0.33) circle (0.2mm);
  \node at (0.5,0.5) {$f$};
  % right loop
  \draw[ultra thick] (1,0) to[bend right=10] (1.8,-0.11);
  \draw[ultra thick] (1.8,-0.11) to[bend right=10] (2.5,-0.1);
  \draw[ultra thick] (2.5,-0.1) to[bend right=20] (2.8,0.6);
  \draw[ultra thick] (2.8,0.6) to[out=110,in=-10] (2.0,1.2);
  \draw[ultra thick] (2.0,1.2) to[out=210,in=45] (1.65,0.6);
  \draw[ultra thick] (1.65,0.6) to (1.25,0.6);
  % left loop
  \draw[ultra thick,gray!60] (0,0) to[bend left=10] (-0.8,-0.1);
  \draw[ultra thick,gray!60] (-0.8,-0.1) to[out=160,in=250] (-1.0,0.8);
  \draw[ultra thick,gray!60] (-1.0,0.8) to[bend left=30] (0,1);
  % pluses and minuses
  \node at (0.5,-0.2) {\tiny $\bm{+}$};
  \node at (0.5,0.08) {\tiny $\bm{-}$};
  \node at (-0.26,0.87) {\tiny $\bm{+}$};
  \node at (-0.04,0.85) {\tiny $\bm{-}$};
  \node at (-0.35,0.5) {\tiny $\bm{+}$};
  \node at (-0.15,0.5) {\tiny $\bm{-}$};
  \node at (-0.26,0.15) {\tiny $\bm{+}$};
  \node at (-0.08,0.18) {\tiny $\bm{-}$};
\end{tikzpicture}
  \end{center}
  \caption{$\phi(\angle(T,T'))+\del_{G^*}(C^*)\in\Jac(G^*)$.}\label{fig:T+cut}
\end{figure}

\vspace{.5cm}

But this sum is exactly $\angle(T^*, T'^*)$.  To see this, root the trees~$T^*$ and~$T'^*$ at a vertex $u$ of
$G^*$ on the cycle in $T^*\cup \{e^*\}$ but different from $f$, as
illustrated in Figure~\ref{fig:dual trees}.
Then the only nonzero vertex angle contributing to $\angle(T^*, T'^*)$ is the angle at the vertex $f$, and by definition, this angle is
$\sum_{i=m}^N (h_i - f)$, as shown in Figure~\ref{fig:angle between duals}.
So we are done.

\end{proof}

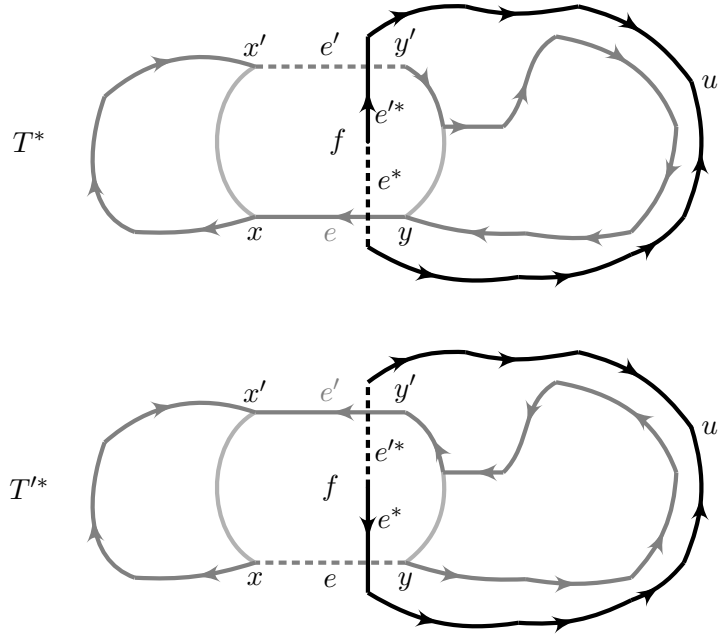
\begin{figure}[!h]
  \begin{center}
% the tree T* rooted at face u not equal to F
\begin{tikzpicture}[scale=2]
  \node at (-1.5,0.5) {$T^*$};
  % edges e and e'
  \node[above] at (0,1) {$x'$};
  \node[above] at (1,1) {$y'$};
  \node[below] at (0,0) {$x$};
  \node[below] at (1,0) {$y$};
  \draw[ultra thick,gray,myarrow] (1,0) to node[midway,below] {$e$} (0,0);
  \draw[ultra thick,gray,densely dashed] (0,1) to node[black,midway,above] {$e'$} (1,1);
  % the face F
  \draw[ultra thick,gray!60] (0,1) to[bend right=60] (0,0);
  \draw[ultra thick,gray,myarrow] (1,1) to[out=-30,in=100] (1.25,0.6);
  \draw[ultra thick,gray!60] (1.25,0.6) to[bend left=30] (1,0);
  \node at (0.55,0.5) {$f$};
  % right loop
  \draw[ultra thick,gray,myarrow] (1.8,-0.11) to[bend left=10] (1,0);
  \draw[ultra thick,gray,myarrow] (2.5,-0.1) to[bend left=10] (1.8,-0.11) ;
  \draw[ultra thick,gray,myarrow] (2.8,0.6) to[bend left=20] (2.5,-0.1);
  \draw[ultra thick,gray,myarrow] (2.0,1.2) to[out=-10,in=110] (2.8,0.6);
  \draw[ultra thick,gray,myarrow] (1.65,0.6) to[out=45,in=210] (2.0,1.2);
  \draw[ultra thick,gray,myarrow] (1.25,0.6) to (1.65,0.6);
  % left loop
  \draw[ultra thick,gray,myarrow] (0,0) to[bend left=10] (-0.8,-0.1);
  \draw[ultra thick,gray,myarrow] (-0.8,-0.1) to[out=160,in=250] (-1.0,0.8);
  \draw[ultra thick,gray,myarrow] (-1.0,0.8) to[bend left=30] (0,1);
  % T^*
  \filldraw (0.75,0.5) circle(0.1mm);
  \draw[ultra thick,densely dashed] (0.75,-0.2) to (0.75,0.5);
  \node at (0.9,0.25) {$e^*$};
  \draw[ultra thick,myarrow] (0.75,0.5) to (0.75,1.2);
  \node at (0.9,0.7) {$e'^*$};
  \draw[ultra thick,myarrow] (0.75,1.2) to[bend left=20] (1.4,1.4);
  \draw[ultra thick,myarrow] (1.4,1.4) to[out=-15,in=190] (2.15,1.4);
  \draw[ultra thick,myarrow] (2.15,1.4) to[bend left=20] (2.9,0.9);
  \node[right] at (2.9,0.9) {$u$};
  \draw[ultra thick,myarrow] (2.85,0.0) to[bend right=20] (2.9,0.9);
  \draw[ultra thick,myarrow] (2.5,-0.25) to[bend right=15] (2.85,0.0);
  \draw[ultra thick,myarrow] (1.75,-0.4) to[bend right=15] (2.5,-0.25);
  \draw[ultra thick,myarrow] (0.75,-0.2) to[bend right=20] (1.75,-0.4);

% the tree T'* rooted face u not equal to F
  \begin{scope}[shift={(0,-2.3)}]
  \node at (-1.5,0.5) {$T'^*$};
  % edges e and e'
  \node[above] at (0,1) {$x'$};
  \node[above] at (1,1) {$y'$};
  \node[below] at (0,0) {$x$};
  \node[below] at (1,0) {$y$};
  \draw[ultra thick,gray,densely dashed] (1,0) to node[black,midway,below] {$e$} (0,0);
  \draw[ultra thick,gray,myarrow] (1,1) to node[midway,above] {$e'$} (0,1);
  % the face F
  \draw[ultra thick,gray!60] (0,1) to[bend right=60] (0,0);
  \draw[ultra thick,gray,myarrow] (1.25,0.6) to[out=100,in=-30] (1,1);
  \draw[ultra thick,gray!60] (1.25,0.6) to[bend left=30] (1,0);
  \node at (0.5,0.5) {$f$};
  % right loop
  \draw[ultra thick,gray,myarrow] (1,0) to[bend right=10] (1.8,-0.11);
  \draw[ultra thick,gray,myarrow] (1.8,-0.11) to[bend right=10] (2.5,-0.1);
  \draw[ultra thick,gray,myarrow] (2.5,-0.1)to[bend right=20] (2.8,0.6);
  \draw[ultra thick,gray,myarrow] (2.8,0.6) to[out=110,in=-10] (2.0,1.2);
  \draw[ultra thick,gray,myarrow] (2.0,1.2) to[out=210,in=45] (1.65,0.6);
  \draw[ultra thick,gray,myarrow] (1.65,0.6) to (1.25,0.6);
  % left loop
  \draw[ultra thick,gray,myarrow] (0,0) to[bend left=10] (-0.8,-0.1);
  \draw[ultra thick,gray,myarrow] (-0.8,-0.1) to[out=160,in=250] (-1.0,0.8);
  \draw[ultra thick,gray,myarrow] (-1.0,0.8) to[bend left=30] (0,1);
  % T'*
  \filldraw (0.75,0.5) circle(0.1mm);
  \draw[ultra thick,myarrow] (0.75,0.5) to (0.75,-0.2);
  \node at (0.9,0.30) {$e^*$};
  \draw[ultra thick,densely dashed] (0.75,0.5) to (0.75,1.2);
  \node at (0.9,0.75) {$e'^*$};
  \draw[ultra thick,myarrow] (0.75,1.2) to[bend left=20] (1.4,1.4);
  \draw[ultra thick,myarrow] (1.4,1.4) to[out=-15,in=190] (2.15,1.4);
  \draw[ultra thick,myarrow] (2.15,1.4) to[bend left=20] (2.9,0.9);
  \filldraw (2.9,0.9) circle (0.1mm);
  \node[right] at (2.9,0.9) {$u$};
  \draw[ultra thick,myarrow] (2.85,0.0) to[bend right=20] (2.9,0.9);
  \draw[ultra thick,myarrow] (2.5,-0.25) to[bend right=15] (2.85,0.0);
  \draw[ultra thick,myarrow] (1.75,-0.4) to[bend right=15] (2.5,-0.25);
  \draw[ultra thick,myarrow] (0.75,-0.2) to[bend right=20] (1.75,-0.4);
\end{scope}
\end{tikzpicture}
  \end{center}
  \caption{Parts of the trees $T^*$ and $T'^*$, rooted at $u$.}~\label{fig:dual trees}
\end{figure}

\bigskip

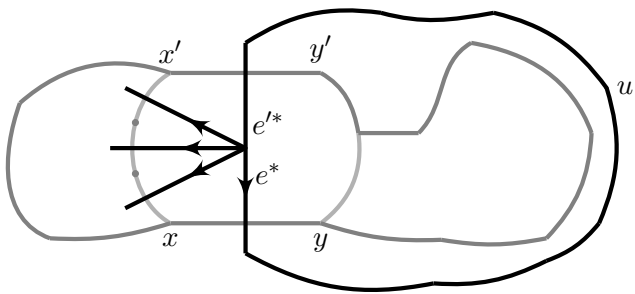
\begin{figure}[!h]
  \begin{center}
% the angle between T* and T'*
\begin{tikzpicture}[scale=2.0]
  % edges e and e'
  \node[above] at (0,1) {$x'$};
  \node[above] at (1,1) {$y'$};
  \node[below] at (0,0) {$x$};
  \node[below] at (1,0) {$y$};
  \draw[ultra thick,gray] (1,0) to (0,0);
  \draw[ultra thick,gray] (1,1) to  (0,1);
  % the face F
  \draw[ultra thick,gray!60] (0,1) to[bend right=60] (0,0);
  \draw[ultra thick,gray] (1.25,0.6) to[out=100,in=-30] (1,1);
  \draw[ultra thick,gray!60] (1.25,0.6) to[bend left=30] (1,0);
  %\node at (0.5,0.5) {$f$};
  % right loop
  \draw[ultra thick,gray] (1,0) to[bend right=10] (1.8,-0.11);
  \draw[ultra thick,gray] (1.8,-0.11) to[bend right=10] (2.5,-0.1);
  \draw[ultra thick,gray] (2.5,-0.1)to[bend right=20] (2.8,0.6);
  \draw[ultra thick,gray] (2.8,0.6) to[out=110,in=-10] (2.0,1.2);
  \draw[ultra thick,gray] (2.0,1.2) to[out=210,in=45] (1.65,0.6);
  \draw[ultra thick,gray] (1.65,0.6) to (1.25,0.6);
  % left loop
  \draw[ultra thick,gray] (0,0) to[bend left=10] (-0.8,-0.1);
  \draw[ultra thick,gray] (-0.8,-0.1) to[out=160,in=250] (-1.0,0.8);
  \draw[ultra thick,gray] (-1.0,0.8) to[bend left=30] (0,1);
  \filldraw (0.5,0.5) circle(0.1mm);
  \draw[ultra thick] (0.5,0.5) to (0.5,1.2);
  \node at (0.65,0.33) {$e^*$};
  \node at (0.65,0.67) {$e'^*$};
  \draw[ultra thick] (0.5,1.2) to[bend left=20] (1.4,1.4);
  \draw[ultra thick] (1.4,1.4) to[out=-15,in=190] (2.15,1.4);
  \draw[ultra thick] (2.15,1.4) to[bend left=20] (2.9,0.9);
  \filldraw (2.9,0.9) circle (0.1mm);
  \node[right] at (2.9,0.9) {$u$};
  \draw[ultra thick] (2.85,0.0) to[bend right=20] (2.9,0.9);
  \draw[ultra thick] (2.5,-0.25) to[bend right=15] (2.85,0.0);
  \draw[ultra thick] (1.75,-0.4) to[bend right=15] (2.5,-0.25);
  \draw[ultra thick] (0.5,-0.2) to[bend right=20] (1.75,-0.4);
  \filldraw[gray] (-0.232,0.67) circle (0.2mm);
  \filldraw[gray] (-0.232,0.33) circle (0.2mm);
  \draw[ultra thick,myarrow] (0.5,0.5) to (0.5,-0.2);
  \draw[ultra thick,myarrow] (0.5,0.5) to (-0.3,0.10);
  \draw[ultra thick,myarrow] (0.5,0.5) to (-0.4,0.5);
  \draw[ultra thick,myarrow] (0.5,0.5) to (-0.3,0.9);
\end{tikzpicture}
  \end{center}
  \caption{$\angle(T^*,T'^*)$.}\label{fig:angle between duals}
\end{figure}

\newpage We now prove our main result.

\begin{proof}[Proof of Theorem~\ref{thm:main}]
Given $[D]\in \Jac(G)$ and $T\in \T(G)$, let $T' = [D]\cdot T$, and let
$T''=\phi([D])\cdot T^*$.  We would like to show that $T''=T'^*$.  By Lemma~\ref{lem:activated edges},
\[
\phi(\angle(T,T')) = \phi([-D]) = \angle(T^*,T'').
\]
By Lemma~\ref{lem:duality},
\[
\phi(\angle(T,T')) = \angle(T^*,T'^*).
\]
Hence, $\angle(T^*,T'')=\angle(T^*,T'^*)$.  Therefore, $\angle(T'',T'^*)=0$, and
the result then follows from Corollary~\ref{cor:zero angle}.
\end{proof}

%\newpage

\end{document}